\documentclass[a4paper,numbook,babel,final,envcountsame,11pt]{article}

\usepackage{amssymb}
\usepackage{amsthm}
\usepackage{amsmath}
\usepackage{graphicx}
\usepackage{array,hhline}
\usepackage{systeme}

\newtheorem{lemma}{Lemma}[section]
\newtheorem{theorem}[lemma]{Theorem}
\newtheorem{proposition}[lemma]{Proposition}

\newtheorem{corollary}[lemma]{Corollary}
\theoremstyle{definition}
\newtheorem{definition}[lemma]{Definition}

\numberwithin{equation}{section}
\numberwithin{figure}{section}

\newcommand{\Lset}{\mathcal{L}}

\begin{document}

\title{\huge On polynomial solutions of PDE}         
\author{Anna R. Gharibyan, Hakop A. Hakopian}        
\date{}          

\maketitle

\begin{abstract}
In this paper we prove that the PDE $p(D)f=q,$ where $p$ and $q$ are multivariate  polynomials, has a solution in the space of polynomials of total degree not exceeding ${n+s},$ where $n$ is the degree of $q$ and $s$ is the zero order of $O=(0,\ldots,0)$ for $p.$
\end{abstract}


\section{Introduction\label{sec:intro}} Let us start with the bivariate case. Denote by $\Pi_n$ the space of biariate polynomials of total degree at most $n:$
$$\Pi_n = \{\sum_{i+j \leq n} a_{ij}x^iy^j\},\quad \dim\Pi_n=1/2(n+1)(n+2).$$
Thus $\Pi_0=\{c: c=const.\}, \dim\Pi_0=1$ and $\Pi_{-1}=\{0\}, \dim\Pi_{-1}=0.$

We set $$\Pi=\cup_{n\ge 0}\Pi_n.$$

Denote also by $\overline\Pi_n$ the space homogeneous polynomials of total degree $n:$
$$\overline\Pi_n = \{\sum_{i+j = n} a_{ij}x^iy^j\},\quad \dim\overline\Pi_n=n+1.$$

For a polynomial $p$ denote by $p(D)$ the respective differentiation operator:
$$p(D)=p(\frac{\partial}{\partial x},\frac{\partial}{\partial y}).$$

\begin{definition}
	Suppose that $p \in \Pi_n,\ p(x, y)=\sum_{i+j \leq n} a_{ij}x^iy^j$. Then  the $k$th homogeneous layer of $p$ we denote by
	$$p_{(k)} = \sum_{i+j = k} a_{ij}x^iy^j.$$
	Denote also by $p^\uparrow$ the upper nonzero homogeneous layer of $p$ and by $p_\downarrow$ the lower nonzero homogeneous layer of $p.$
\end{definition}

\begin{definition}
	The point  $O=(0, 0)$ is called an $m$-fold zero for $p$ if the lower homogeneous layer of $p(x)$ is the $m$th one.
	A point  $a\in \mathbb {C}^2$ is called an $m$-fold zero for $p$ if the lower homogeneous layer of $p(x+a)$ is the $m$th one.
\end{definition}

Let $V$ and $W$ be finite dimensional linear spaces and $\Lset$ be a linear operator $\Lset:V\rightarrow W.$ In the sequel we will use the following formula
\begin{equation}\label{ban}\dim(Im\Lset) = \dim V - \dim(ker \Lset).\end{equation}

Let us bring the following result:
\begin{theorem}[ Th. 5, \cite{H03}; Th. 6, \cite{HT02}]\label{tm:3}
Assume that a polynomial $f_0$ is a solution of the PDE $p(D)f=0.$ Then the partial derivatives of $f_0$ are solutions too.
\end{theorem}

\begin{corollary}\label{cor:1}
	Suppose that $p(x, y)=\sum_{i+j \leq n} a_{ij}x^iy^j$: The homogeneous PDE $p(D)f=0$ has nonzero polynomial solution if and only if $f=1$ is its solution, i.e.,  $p(D)1=a_{00}=0$:
\end{corollary}
In the sequel we will use the following
\begin{theorem}[Th. 2.1, \cite{N}]\label{thm:2}
	Suppose that  the origin $O=(0,0)$ is an $s$-fold zero of $p \in \Pi_n.$ Then the homogeneous PDE $p(D)f=0$ has exactly $D_k$ linearly independent solutions in $\Pi_k,$ where $D_k$ is the $k$th partial sum of the following number series
	$$  \sum_{i=0}^{\infty} d_i = 1 + 2 +\cdots+s +\cdots+ s + \cdots .$$
\end{theorem}
Note that in the case $s=0,$ i.e., $p(0,0)\neq 0,$ this result states that the PDE $p(D)f=0$
has no solution except $f=0.$ Of course this statement coincides with Corollary \ref{cor:1}.

\section{\label{1} Main result}
We are going to prove the following
\begin{theorem}\label{thm:4}
	Assume that   $q \in \Pi_m$ and the origin $O=(0,0)$ is an $s$-fold zero of $p \in \Pi.$ Then we have that the PDE
	\begin{equation*}\label{eq:102}
	p(D)f=q
	\end{equation*}
	has a polynomial solution from $\Pi_{m+s}.$
\end{theorem}
In view of Theorem \ref{thm:2} we get
\begin{corollary}\label{cor:2}
     Assume that  the origin $O=(0,0)$ is an $s$-fold zero of $p \in \Pi$ and $q \in
      \Pi_m.$ Then we have that the solutions of PDE
	\begin{equation}\label{eq:12}
	p(D)f=q,
	\end{equation}
      belonging to $\Pi_{k},\ k\ge m+s,$ form an affine space of dimension $$\sigma = \frac{s(2k-s+1)}{2}.$$
\end{corollary}
Thus the solutions of PDE \eqref{eq:12} can be represented as
$$f = f_0 + \sum_{i=1}^{\sigma} \lambda_i f_i,$$
where $f_0\in\Pi_{m+s}$ is a solution of PDE \eqref{eq:12} and  $f_i,\  i=1,...,\sigma,$ are the linearly independent solutions of the homogeneous PDE $p(D)f=0$  in $\Pi_{k}.$

Let us start the consideration with the case when both $p$ and $q$ are homogeneous. Note that the following is a particular case $s=n$ of Theorem \ref{thm:4}.

\begin{proposition}\label{prp:4}
	Assume that $p \in \overline\Pi_n$ and $q \in \overline\Pi_m.$ Then the PDE \newline $p(D)f=q$
	has a solution from $\overline\Pi_{n+m}.$
\end{proposition}

\begin{proof} Let us verify that in this case the number of linearly independent solutions of PDE $p(D)f=0$ in $\overline\Pi_{n+m}$ is $n.$
	Indeed, according to Theorem \ref{thm:2}, the number of linearly independent solutions of PDE $p(D)f=0$ in $\Pi_{n+m}$ is:
	$$1+2+...+n+mn,$$
	while in $\Pi_{n+m-1}$ is:
	$$1+2+...+n+(m-1)n:$$
	Then consider the linear operator $\Lset:\overline\Pi_{n+m}\rightarrow\overline\Pi_{m}$, given by $\Lset f=p(D)f.$ What we verified above means that $dim(ker\Lset)=n.$  Now, by using the formula \eqref{ban}, we get
	$$dim(Im \Lset) = dim \overline\Pi_{n+m} - dim(ker \Lset)=m + n + 1 - n = m + 1 = dim\overline\Pi_{m}.$$
	Therefore the operator $\Lset$ is on $\overline\Pi_{n+m}$ and the equation \eqref{eq:12} has a solution.
\end{proof}

\begin{proof}[Second proof of Proposition \ref{prp:4}]
By using Theorem \ref{thm:2}, we proved Proposition \ref{prp:4}, which states the polynomial solability of the PDE $p(D)f=q$ in the case when $p$ and $q$ are homogeneous polynomials. Now, let us establish the same result in other way, which will be important in the proof of the result in the case of polynomials of more variables.

Assume that $p \in \overline\Pi_n$ and $q \in\overline \Pi_m:$
$$p(x, y) = a_{0}x^{n} + a_{1}x^{n-1}y + ... + a_{n}y^{n},$$
$$q(x, y) = b_{0}x^{m} + b_{1}x^{m-1}y + ... + b_{m}y^{m}.$$
Let us find a solution $f\in\overline\Pi_{n+m}$ of the PDE
\begin{equation}\label{eq:17}
	p(D)f=q.
	\end{equation}
   Suppose that
$$f(x, y) = \gamma_{0}x^{n+m} + \gamma_{1}x^{n+m-1}y +\cdots + \gamma_{n+m}y^{n+m}.$$
Now, the PDE \eqref{eq:17} looks like
$$(a_{0}x^{n} + a_{1}x^{n-1}y + ... + a_{n}y^{n})(D)(\gamma_{0}x^{n+m} + \gamma_{1}x^{n+m-1}y + ... + \gamma_{n+m}y^{n+m}) $$
$$= b_{0}x^{m} + b_{1}x^{m-1}y + ... + b_{m}y^{m}$$

By equating the coefficients of $x^m, x^{m-1}y,\ldots, y^{m}$ in the left and right hand sides of the equation we get
$$\frac{(n+m)!0!}{m!0!} a_{0}\gamma_{0} + \frac{(n+m-1)!1!}{m!0!} a_{1}\gamma_{1} + ... + \frac{m!n!}{m!0!} a_{n}\gamma_{n} = b_{0},$$
$$\frac{(n+m-1)!1!}{(m-1)!1!} a_{0}\gamma_{1} + \frac{(n+m-2)!2!}{(m-1)!1!} a_{1}\gamma_{2} + ... + \frac{(m-1)!(n+1)!}{(m-1)!1!} a_{n}\gamma_{n+1} = b_{1},$$
$$\vdots$$
$$\frac{n!m!}{0!m!} a_{0}\gamma_{m} + \frac{(n-1)!(m+1)!}{0!m!} a_{1}\gamma_{m+1} + ... + \frac{0!(n+m)!}{0!m!} a_{n}\gamma_{n+m} = b_{m},$$
respectively.

If $a_{0} \neq 0$, then $\gamma_{m+1}, \gamma_{m+2}, ..., \gamma_{m+n},$ are $n$ free variables in above linear system. Thus the solutions of system form an affine space of dimension $n.$

If $a_{0}=0$ and $a_{1} \neq 0$, then it is easily seen that $\gamma_{0}$ becomes a free variable instead of $\gamma_{m+1}$ and again we have $n$ free variables.

In the general case, when $a_{0}=a_1=\cdots a_{k-1}=0$ and $a_{k} \neq 0,$ the
$n$ free variables are $\gamma_{0},\ldots,\gamma_{k-1}$ and $\gamma_{m+k+1},\ldots,\gamma_{m+n}.$
\end{proof}

Thus in the above considered case, when $p \in \overline\Pi_n$ and $q \in \overline\Pi_m,$ we get another proof of the fact that the solutions of the PDE \eqref{eq:17} form an affine space of dimension $n.$

Now consider a particular case $s=0$ of Theorem \ref{thm:4} and Corollary \ref{cor:2}:
\begin{proposition}\label{prp:5}
	Suppose that $p \in \Pi_n,\ q \in \Pi_m$ and the free term of $p$ is not zero: $a_{00} \neq 0.$  Then we have that the PDE
\begin{equation}\label{eq:14}
	p(D)f=q
	\end{equation}	
	has a unique solution $f_0\in\Pi_m.$ The polynomial $f_0$ is the only solution of the PDE \eqref{eq:14} also in each $\Pi_{k},\ k\ge m.$
\end{proposition}

\begin{proof}
		Since $p(D)1=a_{00} \neq 0$, we get from Corollary \ref{cor:1}, that the PDE $p(D)f=0$ has a unique polynomial solution $f=0$:
	
	Then consider the linear operator $\Lset :\Pi_m \rightarrow \Pi_m$ given by $\Lset f=p(D)f.$ Above we verified that $ker\Lset=\{0\}$: Now, by using the formula \eqref{ban}, we get that
	$$dim(Im\Lset) = dim\Pi_m - dim(ker\Lset) = dim\Pi_m.$$
	Thus the equation \eqref{eq:14} has a unique solution in $\Pi_m$ denoted by $f_1.$
	
	Now, concerning the other polynomial spaces, assume by way of contradiction, that the equation \eqref{eq:14} has another solution, denoted by $f_2,\ f_2\neq f_1,$ in $\Pi_k,$ where $k>m.$
	Then we get that $f_1-f_2$ is a nonzero polynomial solution of the PDE $p(D)f=0,$ which contradicts Corollary \ref{cor:1}.
\end{proof}

Now we are in a position to present

\begin{proof}[Proof of Theorem  \ref{thm:4}]
Let us use induction on $m.$  As the first step of induction consider the case $m=-1,\ q \in\Pi_{-1},$ i.e., $q=0$  (the next step is $m=0,\ q\in\Pi_{0},$ i.e., $q=const.\neq 0).$
In this case the PDE $p(D)f=0$ has a polynomial solution $0\in\Pi_{m+s},\ \forall s\ge 0.$

Now assume that the PDE
\begin{equation}\label{eq:16}
	p(D)f=q
	\end{equation}	
has a solution if $q\in\Pi_{m+s-1}.$ Let us prove that it has a solution assuming that $q\in\Pi_{m+s}.$
Now assume that the lower homogeneous layer of $p$ is the $s$th layer: $p_\downarrow=p_{(s)}.$

	Next assume that $f\in\Pi_{m+s}$ and  $f^\uparrow=f_{(m+s)}$ is the upper homogeneous layer of $f.$ Note that the case $f_{(m+s)}=0$ follows from the induction hypothesis.
	
	Consider the following PDE with homogeneous polynomials $p_{(s)}$ and $q_{(m)}:$
		\begin{equation*}\label{eq:9}
	p_{(s)}(D)f = q_{(m)}.
	\end{equation*}
	According to Theorem \ref{prp:4} this equation has a solution $\hat f\in\overline\Pi_{m+s}:$
	\begin{equation}\label{eq:19}
	p_{(s)}(D)\hat f = q_{(m)}.
	\end{equation}
	Next let us seek for a solution of  the PDE \eqref{eq:16} in the form
	$f=g+\hat f:$
		\begin{equation}\label{eq:10}
	p(D)(g+\hat f) = q.
	\end{equation}
	It is easily seen that, in view of \eqref{eq:19}, this PDE is equivalent to
	\begin{equation}\label{eq:11}
	p(D)(g) = r,
	\end{equation}
	where $r=q - p(D)(\hat f).$
	
	Note that $p_\downarrow=p_{(s)}$ readily implies that $p(D)(\hat f)\in\Pi_m.$ Since also $r\in\Pi_m$ we obtain that $r\in\Pi_m$ too.
	
	Next let us verify that $r\in \Pi_{m-1}.$ It is enough to show that $r_{(m)}=0.$  To this purpose, by using \eqref{eq:19}, we obtain that
	
	$r_{(m)}=q_{(m)} - [p(D)(\hat f)]_{(m)}=q_{(m)} - p_{(s)}(D)(\hat f)=q_{(m)} -q_{(m)} =0.$

	Now, in view of $r\in\Pi_{m-1},$ let us use the induction hypothesis. Hence, we get    that the PDE \eqref{eq:11} has a solution denoted by $g_0,$ where $g_0\in \Pi_{m+s-1}.$ Therefore $f=g_0+\hat f$ is a solution of the PDE \eqref{eq:16}. Note also that this implies that $\hat f=f_{m+s}.$
\end{proof}

\section{The case of more than two variables}

Let us use stamdard multivariate notation.

For $x=(x_1,\ldots,x_k)\in \mathbb C^k$ and $\alpha=(\alpha_1,\ldots,\alpha_k\in mathbb Z_+^k$ denote
$$x^\alpha=x_1^{\alpha_1}\cdots x_k^{\alpha_k}, \quad |\alpha|=\alpha_1+\cdots+\alpha_k,\quad  \alpha!=\alpha_1!\cdots\alpha_k!.$$

Denote by $\Pi_n^{(k)}$ the space polynomials of $k$ variables of total degree at most $n,$ 
$$\Pi_n = \{\sum_{|\alpha|\le n}a_\alpha x^\alpha\},\quad\dim\Pi_n^{(k)}=\binom{n+k}{k}.$$

We set $$\Pi^{(k)}=\cup_{n\ge 0}\Pi_n^{(k)}.$$

Denote also by $\overline\Pi_n^{(k)}$ the space homogeneous polynomials of total degree $n,$ for which we have that
$$\overline\Pi_n = \{\sum_{|\alpha|= n}a_\alpha x^\alpha\},\quad\dim\overline\Pi_n^{(k)}=\binom{n+k-1}{k-1}.$$

The following result holds for polynomials of $k$ variables:
\begin{theorem}\label{thm:4}
	Assume that $q \in \Pi_m^{(k)}$ and the origin $O=(0,\ldots,0)$ is an $s$-fold zero of $p \in \Pi^{(k)}.$ Then we have that the PDE
	\begin{equation*}\label{eq:102}
	p(D)f=q
	\end{equation*}
	has a polynomial solution from $\Pi_{m+s}^{(k)}.$
\end{theorem}

It is clear from the previous section that all we need is  to generalize Proposition \ref{prp:4} to the case of $k$ variables:
\begin{proposition}\label{prp:4k}
	Assume that $p \in \overline\Pi_n^{(k)}$ and $q \in \overline\Pi_m^{(k)}.$ Then the PDE
	\begin{equation} p(D)f=q\label{eq:99}
	\end{equation}
	has a solution from $\overline\Pi_{n+m}^{(k)}.$
\end{proposition}

\begin{proof} 
Assume that $p \in \overline\Pi_n$ and $q \in\overline \Pi_m:$
$$p(x)=\sum_{|\alpha|= n}a_\alpha x^\alpha,$$
$$q(x)=\sum_{|\alpha| = m}b_\alpha x^\alpha.$$
Let us find a solution $f\in\overline\Pi_{n+m}$ of the PDE
\begin{equation}\label{eq:17}
	p(D)f=q.
	\end{equation}
   Suppose that
$$f(x)=\sum_{|\alpha|= m+n}\gamma_\alpha x^\alpha.$$
Now, the PDE \eqref{eq:17} looks like

$$\sum_{|\alpha|= n}a_\alpha x^\alpha(D)\sum_{|\alpha|= m+n}\gamma_\alpha x^\alpha= \sum_{|\alpha|= m}b_\alpha x^\alpha.$$

By equating the terms with $x^{\alpha_0}, \ |\alpha_0|\le m,$ in the left and right hand sides of the equation, we get that

$$\sum_{|\alpha|= n}a_\alpha x^\alpha(D)\gamma_{\alpha+\alpha_0} x^{\alpha+\alpha_0}= b_{\alpha_0} x^{\alpha_0}.$$

For the respective coefficients we have that
\begin{equation}\label{eq:100}\sum_{|\alpha|= n}(\alpha+\alpha_0)!a_\alpha  \gamma_{\alpha+\alpha_0}= \alpha_0! b_{\alpha_0}.\end{equation}

By the change of unknowns $\gamma'_\alpha=\alpha!\gamma_\alpha$ we get the linear system
$$\sum_{|\alpha|= n}a_\alpha  \gamma_{\alpha+\alpha_0}'= \alpha_0! b_{\alpha_0} .$$
Let us verify that the main matrix of this system has full rank.

Note that in the row of the matrix with $\alpha_0=(0,\ldots,0)$ we have the coefficients of the polynomial $p(x),$ while in the row with any $\alpha_0,$ where $|\alpha_0|\le m,$ we have the coefficients of the polynomial $x^\alpha p(x).$ 

Thus if the rows of the main matrix are linearly dependent then so are the polynomials 
$x^\alpha p(x),\ |\alpha_0|\le m,$ which evidently takes place only if $p(x)\equiv 0.$

Therefore we have that the linear system \eqref{eq:100} also is a full rank linear system, where we have
$\binom{m+k-1}{k-1}$ equations and  $\binom{n+m+k-1}{k-1}$ unknowns.

Thus we have that the system is consistent. Moreover, in the solution of the linear system we have exactly
$$ \sigma(n,m,k):=\binom{n+m+k-1}{k-1}-\binom{m+k-1}{k-1}$$
free variables.
\end{proof}

\renewcommand{\refname}{REFERENCES}

\end{document}